\documentclass{amsart}[11pt]

\usepackage{amssymb,mathrsfs,amscd,graphicx,array, multirow,longtable,
enumerate, amsfonts, euscript}

\usepackage{hyperref}
\makeatletter
\def\@tocline#1#2#3#4#5#6#7{\relax
  \ifnum #1>\c@tocdepth 
  \else
    \par \addpenalty\@secpenalty\addvspace{#2}%
    \begingroup \hyphenpenalty\@M
    \@ifempty{#4}{%
      \@tempdima\csname r@tocindent\number#1\endcsname\relax
    }{%
      \@tempdima#4\relax
    }%
    \parindent\z@ \leftskip#3\relax \advance\leftskip\@tempdima\relax
    \rightskip\@pnumwidth plus4em \parfillskip-\@pnumwidth
    #5\leavevmode\hskip-\@tempdima
      \ifcase #1
       \or\or \hskip 1em \or \hskip 2em \else \hskip 3em \fi%
      #6\nobreak\relax
    \dotfill\hbox to\@pnumwidth{\@tocpagenum{#7}}\par
    \nobreak
    \endgroup
  \fi}
\makeatother

\usepackage{letltxmacro}
\LetLtxMacro{\oldsqrt}{\sqrt}
\renewcommand{\sqrt}[2][]{\,\oldsqrt[#1]{#2}\,}

\def\bmu{\boldsymbol \mu}

\newcommand{\abs}[1]{\lvert #1 \rvert}
\newcommand{\zmod}[1]{\mathbb{Z}/ #1 \mathbb{Z}}

\newcommand{\dangle}[1]{\left\langle #1 \right\rangle}

\DeclareMathSymbol{\twoheadrightarrow} {\mathrel}{AMSa}{"10}

\DeclareMathOperator{\Gal}{Gal}

\DeclareMathOperator{\Tr}{Tr}
\DeclareMathOperator{\Nm}{N}  

\def\a{{\mathfrak a}} 
\def\d{{\mathfrak d}} 

\newcommand{\f}{\mathfrak{f}}

\newcommand{\p}{\mathfrak{p}}

\newcommand{\ff}{\mathbb{F}}

\newcommand{\nn}{\mathbb{N}}
\newcommand{\oo}{\mathbb{O}}

\newcommand{\qq}{\mathbb{Q}}

\newcommand{\zz}{\mathbb{Z}}

\newcommand{\OO}{\mathcal{O}}

\newcommand{\calO}{\mathcal{O}}

\DeclareMathOperator{\Mass}{Mass}

\DeclareMathOperator{\Ell}{Ell}

\newcommand{\grp}{\mathfrak{p}}

\newcommand{\grf}{\mathfrak{f}}

\def\makeop#1{\expandafter\def\csname#1\endcsname
  {\mathop{\rm #1}\nolimits}\ignorespaces}
\makeop{Hom}   \makeop{End}   \makeop{Aut}   \makeop{Isom}  \makeop{Pic} 
\makeop{Gal}   \makeop{ord}   \makeop{Char}  \makeop{Div}   \makeop{Lie} 
\makeop{PGL}   \makeop{Corr}  \makeop{PSL}   \makeop{sgn}   \makeop{Spf}
\makeop{Spec}  \makeop{Tr}    \makeop{Nr}    \makeop{Fr}    \makeop{disc}
\makeop{Proj}  \makeop{supp}  \makeop{ker}   \makeop{im}    \makeop{dom}
\makeop{coker} \makeop{Stab}  \makeop{SO}    \makeop{SL}    \makeop{SL}
\makeop{Cl}    \makeop{cond}  \makeop{Br}    \makeop{inv}   \makeop{rank}
\makeop{id}    \makeop{Fil}   \makeop{Frac}  \makeop{GL}    \makeop{SU}
\makeop{Nrd}   \makeop{Sp}    \makeop{Tr}    \makeop{Trd}   \makeop{diag}
\makeop{Res}   \makeop{ind}   \makeop{depth} \makeop{Tr}    \makeop{st}
\makeop{Ad}    \makeop{Int}   \makeop{tr}    \makeop{Sym}   \makeop{can}
\makeop{length}\makeop{SO}    \makeop{torsion} \makeop{GSp} \makeop{Ker}
\makeop{Adm}   \makeop{Mat}
\def\makebb#1{\expandafter\def
  \csname bb#1\endcsname{{\mathbb{#1}}}\ignorespaces}
\def\makebf#1{\expandafter\def\csname bf#1\endcsname{{\bf
      #1}}\ignorespaces} 
\def\makegr#1{\expandafter\def
  \csname gr#1\endcsname{{\mathfrak{#1}}}\ignorespaces}
\def\makescr#1{\expandafter\def
  \csname scr#1\endcsname{{\EuScript{#1}}}\ignorespaces}
\def\makecal#1{\expandafter\def\csname cal#1\endcsname{{\mathcal
      #1}}\ignorespaces} 

\def\doLetters#1{#1A #1B #1C #1D #1E #1F #1G #1H #1I #1J #1K #1L #1M
                 #1N #1O #1P #1Q #1R #1S #1T #1U #1V #1W #1X #1Y #1Z}
\def\doletters#1{#1a #1b #1c #1d #1e #1f #1g #1h #1i #1j #1k #1l #1m
                 #1n #1o #1p #1q #1r #1s #1t #1u #1v #1w #1x #1y #1z}
\doLetters\makebb   \doLetters\makecal  \doLetters\makebf
\doLetters\makescr 
\doletters\makebf   \doLetters\makegr   \doletters\makegr

\def\Mass{{\rm Mass}}

\newcounter{thmcounter} 
\numberwithin{thmcounter}{section}
\newtheorem{lem}[thmcounter]{Lemma}

\newtheorem{prop}[thmcounter]{Proposition}
\theoremstyle{definition}

\newtheorem{sect}[thmcounter]{}

\numberwithin{equation}{section}

\newtheoremstyle{notitle}  
  {}
  {}
  {\itshape}
  {}
  {}
  {\ }
  {.5em}
  {}
\theoremstyle{notitle}

\title[Numerical invariants of imaginary quadratic orders]{Numerical
  Invariants of Totally Imaginary Quadratic $\zz[\sqrt{p}]$-orders}
\author{Jiangwei Xue,Tse-Chung Yang and Chia-Fu Yu}

\address{(Xue) Collaborative Innovation Centre of Mathematics, 
School of Mathematics and Statistics, Wuhan University, Luojiashan,
Wuhan, Hubei, 430072, P.R. China.}
\email{xue\_j@whu.edu.cn}

\address{(Yang) Institute of Mathematics, Academia Sinica,
  Astronomy-Mathematics Building, 6F, No. 1, Sec. 4, Roosevelt Road,
  Taipei 10617, TAIWAN.} 
\email{tsechung@math.sinica.edu.tw}

\address{(Yu) Institute of Mathematics,
  Academia Sinica and NCTS, Astronomy-Mathematics
  Building, No. 1, Sec. 4, Roosevelt Road, Taipei 10617, TAIWAN.}
\email{chiafu@math.sinica.edu.tw} \address{
  The Max-Planck-Institut f\"ur Mathematik \\
  Vivatsgasse 7, Bonn \\
  Germany 53111} \email{chiafu@mpim-bonn.mpg.de}

\begin{document}

\date{\today}  

\begin{abstract}
  Let $A$ be a real quadratic order of discriminant $p$ or $4p$ with 
  a prime $p$. In this paper we classify all proper 
  totally imaginary quadratic $A$-orders
  $B$ with index $w(B)=[B^\times: A^\times]>1$. 
  We also calculate numerical invariants of these
  orders including the class number, the index $w(B)$  
  and the numbers of local optimal embeddings of these orders into
  quaternion orders. These numerical invariants are useful for
  computing the class numbers of totally definite quaternion 
  algebras.  
\end{abstract}

\maketitle




\section{Introduction}
\label{sec:intro}

Let $F$ be a totally real number field with the ring of integers
$O_F$. Let $D$ be a totally definite
quaternion algebra over $F$ and $\calO\subset D$ an $O_F$-order in 
$D$. A main interest in the arithmetic of quaternion algebras is to
compute the class number $h(\calO)$ 
of $\calO$ (for locally free ideal classes).
Eichler's class number formula 
states that
\begin{equation}
  \label{eq:class_f}
  h(\calO)=\Mass(\calO)+\Ell(\calO),
\end{equation}
where $\Mass(\calO)$ is the mass of $\calO$, 
which is (by definition) a weighted
sum over all the ideal classes of $\calO$, and $\Ell(\calO)$ is the
elliptic part of $h(\calO)$, which is expressed as follows:
\begin{equation}
  \label{eq:ell_f}
  \Ell(\calO)=\frac{1}{2} \sum_{w(B)>1} h(B) 
(1-w(B)^{-1}) \prod_{\grp} m_\grp(B).
\end{equation}      
In the summation 
$B$ runs through all (non-isomorphic) quadratic
$O_F$-orders such that the field $K$ of fractions 
can be embedded into $D$
and the index $w(B):=[B^\times:O_F^\times]>1$. The symbol $h(B)$ 
denotes the class number of $B$, and for any finite prime $\grp$ of $F$,
$m_\grp(B)$ is the number of equivalence classes of optimal embeddings
of $B_\grp:=B\otimes_{O_F} O_{F_\grp}$ into 
$\calO_\grp:=\calO \otimes_{O_F} O_{F_\grp}$. We refer to 
Eichler \cite{eichler:crelle55}, Vigneras \cite[Chapter V, Corollary 2.5,
p. 144]{vigneras} and K\"orner~\cite[Theorem 2]{korner:1987}) for more
details. 

One can use the mass formula 
(cf. \cite[Chapter V, Corollary 2.3]{vigneras} and 
\cite[Section 5]{xue-yang-yu:ECNF}) to compute $\Mass(\calO)$. 
When the order $\calO$ is not too complicated, for
example if $\calO$ is an Eichler order, 
the computation of numbers of local 
optimal embeddings is known by Eichler (cf.~\cite[p.~94]{vigneras}) and
Hijikata \cite[Theorem 2.3, p. 66]{hijikata:JMSJ1974}. 
Also see Pizer \cite[Sections 3-5]{pizer:2} for some extensions. 
A major difficulty in adapting Eichler's class number formula 
is to find all the quadratic $O_F$-orders $B$ with the properties
stated below (\ref{eq:ell_f}). 
It is not hard to see that the fraction field $K$ of $B$ must be 
totally imaginary over $F$ and the information whether $K$ can be
embedded into $D$ is already contained in local optimal embeddings.   

In this paper we classify all totally imaginary quadratic $O_F$-orders
$B$ with $w(B)>1$ in the case where $F=\bbQ(\sqrt{p})$ is a real
quadratic field with a prime number $p$. We also compute the class
number $h(B)$ and the index $w(B)$ of them. As a consequence of our
computations we obtain a formula for $h(\calO)$ for any Eichler order
$\calO$ of square-free level in an arbitrary totally definite
quaternion algebra over $\bbQ(\sqrt{p})$ (see Section~\ref{sec:hO}). 

Our motivation of computing the class number of quaternion orders
comes from the study of supersingular abelian surfaces over finite
fields. We are interested in finding an explicit formula for the
number $H(p)$ of isomorphism classes of (necessarily superspecial) 
abelian surfaces in the
isogeny class over the prime field $\bbF_p$ 
corresponding to the Weil $p$-number $\sqrt{p}$. 
The endomorphism algebras of these abelian varieties are isomorphic to 
the totally
definite quaternion algebra $D_{\infty_1,\infty_2}$ over 
$F=\bbQ(\sqrt{p})$ which is ramified only at the two real places. 
When $p=2$ or $p\equiv 3 \pmod 4$, the number $H(p)$ is equal to the
class number $h(\oo_1)$ of a maximal order $\oo_1$ in
$D_{\infty_1,\infty_2}$. When $p=1\pmod
4$, we show that $H(p)=h(\oo_1)+h(\oo_8)+h(\oo_{16})$, 
where $\oo_8$ and $\oo_{16}$ are certain proper 
$A=\zz[\sqrt{p}]$-suborders 
of $\oo_1$ of index $8$ and $16$, respectively. 
(We say  $\calO$ is a  ``proper'' $A$-order if
$\calO \cap F=A$.)
For the non-maximal cases the generalized class number formula
\cite[Theorem 1.5]{xue-yang-yu:ECNF}
requires to find all totally imaginary proper quadratic
$A$-orders $B$ with $w(B):=[B^\times:
A^\times]>1$ and compute the numerical invariants $h(B)$ and $w(B)$
again. 
These technical issues are dealt within this
paper. The results of this paper 
will be used in \cite{xue-yang-yu:ECNF} to compute the number $H(p)$ of
superspecial abelian surfaces.
See \cite[Theorem 1.2]{xue-yang-yu:ECNF} for the final formula for 
$H(p)$.     

The paper is organized as
follows. 
Section~\ref{sec:units-totally-imag} classifies all totally imaginary
quadratic fields $K$ over $F=\qq(\sqrt{p})$ with $w_K:=[O_K^\times:
O_F^\times]>1$. We express the class numbers $h(K)$ of these fields $K$ 
in terms of $h(F)$ and compute $w_K$. 
Section~\ref{sec:OF-orders} classifies all $O_F$-orders $B$ in $K$
with $w(B)>1$. We also compute the numerical invariants $h(B)$ and $w(B)$
of these orders. Section~\ref{sec:quadr-prop-zzsqrtp-orders} classifies 
all proper $A$-orders $B$ in $K$ with $w(B)>1$ when $p\equiv 1\pmod 4$.
We compute the same numerical invariants of them and the numbers of 
related local optimal embeddings mentioned above.   

\section{Totally imaginary quadratic extensions
  $K/F$}\label{sec:units-totally-imag}
\numberwithin{thmcounter}{section} In this section, we classify all
the totally imaginary quadratic extensions of $\qq(\sqrt{p})$ that
have strictly larger groups of units than $O_{\qq(\sqrt{p})}^\times$.
Throughout this section, $F$ denotes a totally real number field with
ring of integers $O_F$ and group of units $O_F^\times$, and $K$ always
denotes a totally imaginary quadratic extension of $F$. We write
$\bmu_K$ for the torsion subgroup of $O_K^\times$.  It is a finite
cyclic subgroup of $O_K^\times$ consisting of all the roots of unity
in $K$. Clearly, $\bmu_F=\{\pm 1\}$. The quotient groups
$O_F^\times/\bmu_F$ and $O_K^\times/\bmu_K$ are free abelian groups of
rank $[F:\qq]-1$ by the Dirichlet's Unit Theorem (cf. \cite[Theorem
I.7.4]{MR1697859}). 

\begin{sect}\label{subsec:intro-of-j-K-over-F}
  Since the free abelian groups $O_F^\times/\bmu_F$ and
  $O_K^\times/\bmu_K$ have the same rank, the natural embedding
$O_F^\times/\bmu_F\hookrightarrow  O_K^\times/\bmu_K  $
realizes $O_F^\times/\bmu_F$ as a subgroup of $O_K^\times/\bmu_K$ of finite index, called the Hasse unit index,
\begin{equation}
  \label{xeq:47}
Q_{K/F}:=[O_K^\times/\bmu_K: O_F^\times/\bmu_F]=[O_K^\times:\bmu_KO_F^\times].  
\end{equation}
In particular, $O_F^\times$ has finite index in $O_K^\times$.

Suppose that $\bmu_K=\dangle{\zeta_{2n}}$, where $\zeta_{2n}$ is a
primitive $2n$-th root of unity. Let $\iota: x\mapsto \iota(x)$ be the unique
nontrivial element of $\Gal(K/F)$.  By \cite[Theorem
4.12]{Washington-cyclotomic},
$Q_{K/F}$ is either 1 or 2. This can be seen in the following
way. There is a homomorphism $\phi_K$ whose image
contains $\bmu_K^2=\phi_K(\bmu_K)$:
\begin{equation}
  \label{xeq:51}
  \phi_K: O_K^\times\to \bmu_K, \qquad u\mapsto u/\iota(u).
\end{equation} One easily checks that
$\phi_K(u)\in \bmu_K^2$ if and only if $u\in \bmu_KO_F^\times$, hence
$Q_{K/F}=[\phi_K(O_K^\times):\bmu_K^2]\leq 2$. Moreover,  $Q_{K/F}=2$ if
and only if $\phi_K$ is surjective, i.e. there exists $z\in O_K^\times$ such that
\begin{equation}
  \label{xeq:52}
  z=\iota(z)\zeta_{2n}.
\end{equation}
We note that (\ref{xeq:51}) also implies that 
\begin{equation}\label{xeq:55}
  u^2\equiv \Nm_{K/F} (u) \pmod{\bmu_K}, \quad \forall u\in O_K^\times. 
\end{equation}

  Consider the quotient group $O_K^\times/O_F^\times$. If $Q_{K/F}=1$, then
$O_K^\times=\bmu_K O_F^\times$, and 
\begin{equation}\label{xeq:15} 
O_K^\times/O_F^\times\cong \bmu_K/\bmu_F=
\bmu_K/\{\pm 1\},
\end{equation} 
which is a cyclic group of order $n$ generated by the image of
$\zeta_{2n}$. 
If $Q_{K/F}=2$,  there is an exact sequence 
\begin{equation}\label{xeq:54}
1\to (\bmu_KO_F^\times)/O_F^\times\to O_K^\times/O_F^\times\to \bmu_K/\bmu_K^2\to 1.  
\end{equation}
Let $z\in O_K^\times$ be an element satisfying (\ref{xeq:52}). Then 
\begin{equation}\label{xeq:53}
  z^2=\Nm_{K/F}(z)\zeta_{2n},
\end{equation}
so $\zeta_{2n}\equiv z^2 \pmod{O_F^\times}$. Therefore, $O_K^\times/O_F^\times$ is a cyclic group of order $2n$ generated by the image of
$z$ in this case. Either way, $O_K^\times/O_F^\times$ is a cyclic group. Its order 
$w_K:=\abs{O_K^\times/O_F^\times}$ is given by 
\begin{equation}\label{xeq:14}
  w_K=\frac{1}{2}\abs{\bmu_K}\cdot Q_{K/F}=
  \begin{cases}
    \abs{\bmu_K}/2 &\qquad \text{ if } \quad Q_{K/F}=1;\\
    \abs{\bmu_K}  &\qquad \text{ if } \quad Q_{K/F}=2.\\
  \end{cases}
\end{equation}
\end{sect}

For the rest of this section, we assume that $F=\qq(\sqrt{d})$ is a
real quadratic field with square free $d\in \nn$. We will soon
specialize further to the case that $F=\qq(\sqrt{p})$ with a prime
$p\in \nn$. Recall that
\[O_F=
\begin{cases}
\zz\left[(1+\sqrt{d})/2\right] &\qquad   \text{ if } \quad d\equiv 1\phantom{,2}
\pmod{4}; \\
\zz[\sqrt{d}] &\qquad   \text{ if } \quad d\equiv 2, 3
\pmod{4}.\\
\end{cases}
\]
The \textit{fundamental unit} by definition is the unit
$\epsilon\in O_F^\times$ such that $O_F^\times=\{\pm \epsilon^a\mid
a\in \zz\}$ and 
$\epsilon>1$.  Note that $\epsilon$ is totally positive if and only if
$\Nm_{F/\qq}(\epsilon)=1$.


\begin{lem}\label{lem:norm-fund-unit-negative}
  Let $\epsilon$ be the fundamental unit of $F=\qq(\sqrt{d} )$, and
  $K$ a totally imaginary quadratic extension of $F$ with
  $\bmu_K=\dangle{\zeta_{2n}}$.  The index $Q_{K/F}=2$ if and only if
  $\Nm_{F/\qq}(\epsilon)=1$ and the equation
  \begin{equation}\label{xeq:12}
z^2= \epsilon\, \zeta_{2n}    
  \end{equation}
  has a solution in $K$.  In particular, if
  $\Nm_{F/\qq}(\epsilon)=-1$, then $Q_{K/F}=1$.
\end{lem}
\begin{proof}
  Only the first statement needs to be proved, as the second one
  follows easily.  The sufficiency is obvious. We prove the ``only
  if'' part.  Suppose that $Q_{K/F}=2$.  Let $z\in O_K^\times$ be a
  representative of a generator of $O_K^\times/\bmu_K\cong \zz$. By
  (\ref{xeq:55}), $O_F^\times/\bmu_F$ can be generated by a totally
  positive unit, namely $\Nm_{K/F}(z)$.  Therefore, $\epsilon$ must be
  totally positive, which happens if and only if
  $\Nm_{F/\qq}(\epsilon)=1$. Replacing $z$ by $1/z$ if necessary, we
  may assume $\Nm_{K/F}(z)=\epsilon$. By (\ref{xeq:54}), there exists
  an odd number $2c+1\in \zz$ such that $z=\iota(z)\zeta_{2n}^{2c+1}$.
  We further replace $z$ by $z\zeta_{2n}^{-c}$, then it satisfies
  equation (\ref{xeq:12}).
\end{proof}






\begin{sect}\label{subsec:list-of-roots-of-unity}
  Since $[K:\qq]=4$, we have $\varphi(2n)\leq 4$. The possible $n$'s
  are $1, 2,3, 4, 5, 6$.  Moreover, the cases $n=4, 5, 6$ can only
  happen in the following situations:
\begin{itemize}
\item if $n=4$, then $K=\qq(\zeta_8)=\qq(\sqrt{-1},\sqrt{2})$ and
  $F=\qq(\sqrt{2})$; 
\item if $n=5$, then $K=\qq(\zeta_{10})$ and   $F=\qq(\sqrt{5})$;
\item if $n=6$, then $K=\qq(\zeta_{12})=\qq(\sqrt{3}, \sqrt{-1})$ and
  $F=\qq(\sqrt{3})$. 
\end{itemize}
\end{sect}


\begin{lem}\label{lem:norm}
  Let $\epsilon$ be the fundamental unit of $F=\qq(\sqrt{p})$, where
  $p\in \nn$ is a prime number. Then $\Nm_{F/\qq}(\epsilon)=1$ if and
  only if $p\equiv 3\pmod{4}$.
\end{lem}
\begin{proof}
  If $p=2$, then $\epsilon=1+\sqrt{2}$, so
  $\Nm_{F/\qq}(\epsilon)=-1$.  By \cite[Corollary
  18.4bis, p. 134]{MR963648}, if $p\equiv 1 \pmod{4}$, the norm of the
  fundamental unit is $-1$. On the other hand, if $p\equiv 3
  \pmod{4}$, we claim that $\Nm_{F/\qq}(u)=1$ for any $u\in O_F^\times$.  Indeed,
  If $u=a+b\sqrt{p}$ has norm $-1$, then $a^2-b^2p=-1$.
  Modulo $p$ on both sides, we see that $-1$ is a square in
  $\zmod{p}$, contradicting to the assumption $p\equiv
  3\pmod{4}$.
\end{proof}

\begin{prop}\label{prop:sqrt-epsilon-and-sqrt-2}
  Suppose that $p\equiv 3 \pmod{4}$, and $\epsilon$ is the fundamental
  unit of $F=\qq(\sqrt{p})$. Then $\sqrt{\epsilon/2}\in F$, and
  $\sqrt{\epsilon/2} \equiv (1+\sqrt{p})/2 \pmod{O_F}$.
\end{prop}
\begin{proof}
  It is known that $\epsilon=2x^2$ for some $x\in F$ when $p\equiv 3
  \pmod{4}$ (cf. \cite[Lemma 3, p. 91]{MR1344833} or \cite[Lemma
  3.2(1)]{MR3157781}).  We have $(2x)^2=2\epsilon\equiv 0 \pmod{
    2O_F}$. Clearly, $2x\in O_F$ but $x\not \in O_F$. On the other
  hand, $1+\sqrt{p}$ is the only nonzero nilpotent element in
  $O_F/2O_F$. So we must have $2x\equiv 1+\sqrt{p} \pmod{2O_F}$, and
  the second part of the proposition follows.
\end{proof}

\begin{prop}\label{prop:ring-of-int-in-quad-ext}
  Suppose that $p\equiv 3 \pmod{4}$.  Let $\epsilon$ be the (totally
  positive) fundamental unit of $F=\qq(\sqrt{p})$, and
  $K=F(\sqrt{-\epsilon})$. Then
  $K=F(\sqrt{-2})=\qq(\sqrt{p},\sqrt{-2})$, and 
  $O_K=\zz[\sqrt{p}, \sqrt{-\epsilon}]$.
\end{prop}
\begin{proof}
  By Proposition~\ref{prop:sqrt-epsilon-and-sqrt-2}, $K=\qq(\sqrt{p},
  \sqrt{-2})$.  Let $B:=\zz[\sqrt{p},
  \sqrt{-\epsilon}]=O_F[\sqrt{-\epsilon}]\subseteq O_K$, and
  $\d_B=\d_{B/\zz}$ be the discriminant of $B$ with respect to
  $\zz$. To show that $B=O_K$, it is enough to show that $\d_B$
  coincides with $\d_{O_K}=\d_K$, the absolute discriminant of $K$. We
  have $\d_K=4p\cdot (-8)\cdot (-8p)=2^8p^2$ by Exercise 42(f) of
  \cite[Chapter 2]{MR0457396}. On the other hand, 
   \[\d_B=\d_F^2\cdot \Nm_{F/\qq}(\d_{B/O_F})=(4p)^2\cdot \Nm_{F/\qq}(-4\epsilon)=2^8p^2=\d_K. \]
   So indeed $O_K=\zz[\sqrt{p}, \sqrt{-\epsilon}]$.
\end{proof}

The following proposition determines $Q_{K/F}$ for any totally
imaginary quadratic extension $K$ of $F=\qq(\sqrt{p})$.

  \begin{prop}\label{prop:criterion-for-jK}
    Suppose $F=\qq(\sqrt{p})$. Then $Q_{K/F}=2$ if and only if $p\equiv 3
    \pmod{4}$,  and $K$ is either $F(\sqrt{-1})=\qq(\sqrt{p},
    \sqrt{-1})$ or $F(\sqrt{-\epsilon})=\qq(\sqrt{p}, \sqrt{-2})$. 
  \end{prop}
  \begin{proof}
    By Lemma~\ref{lem:norm-fund-unit-negative} and Lemma~\ref{lem:norm}, $Q_{K/F}=1$ for all $K$
    if $p=2$ or $p\equiv 1\pmod{4}$.  Assume that $p\equiv 3 \pmod{4}$
    for the rest of the proof. Combining
    Lemma~\ref{lem:norm-fund-unit-negative} and
    Proposition~\ref{prop:sqrt-epsilon-and-sqrt-2}, we see that
    $Q_{K/F}=2$ if and only if the equation
\begin{equation}\label{xeq:13}
  y^2=2\zeta_{2n}
\end{equation}
has a solution in $K$. By
Section~\ref{subsec:list-of-roots-of-unity}, the possible values of
$n$ are $6, 3, 2, 1$. 

If $n=6$, then $p=3$ and $K=\qq(\zeta_{12})=\qq(\sqrt{3},
\sqrt{-1})$. We claim that $\qq(\sqrt{2}\zeta_{24})=K$. Indeed,
$\qq(\sqrt{2}\zeta_{24})=\qq(\zeta_3, \sqrt{2}\zeta_8)$. Since
$\zeta_8=\frac{\sqrt{2}}{2}+\frac{\sqrt{-2}}{2}$, our claim follows.
Therefore, (\ref{xeq:13}) has a solution in $K$ and $Q_{K/F}=2$ in
this case.

Assume that $p>3$ for the rest of the proof. 

If $n=3$, then $K=\qq(\sqrt{p}, \sqrt{-3})$. If $\sqrt{2}\zeta_{12}\in
K$, then it implies that $\sqrt{-2}=\sqrt{2}\zeta_4\in K$, which is
clearly false. Therefore, $Q_{K/F}=1$ if $K=\qq(\sqrt{p}, \sqrt{-3})$ with
$p>3$.

If $n=2$, then $K=\qq(\sqrt{p}, \sqrt{-1})$. We have $(1+\sqrt{-1})^2=
2\sqrt{-1}=2\zeta_4$. Therefore, $Q_{K/F}=2$ in this case. 

Lastly, suppose that $n=1$. Then $Q_{K/F}=2$ implies that
$K=F(\sqrt{-2})=\qq(\sqrt{p}, \sqrt{-2})$. One easily checks that
$\bmu_K$ is indeed $\{\pm 1\}$ so this is also  sufficient for $Q_{K/F}=2$.
\end{proof}

In the case where $F=\qq(\sqrt{d})$ is an arbitrary real quadratic
field and $K$ is an imaginary bicyclic biquadratic field containing
$F$, the calculation of $Q_{K/F}$ is discussed in \cite[Section
2]{MR0441914}.

\begin{sect}\label{subsec:quotient-units-gp-cyclic}
  The following table gives a complete list of the extensions
  $K/\qq(\sqrt{p})$ with $w_K=[O_K^\times:O_{\qq(\sqrt{p})}^\times]>1$ for all
  primes $p$.

\bigskip

\renewcommand{\arraystretch}{1.3}
\noindent\begin{tabular}{|>{$}c<{$}|>{$}c<{$}|>{$}c<{$}||>{$}c<{$}|>{$}c<{$}|>{$}c<{$}||>{$}c<{$}|>{$}c<{$}|>{$}c<{$}|}
\hline
p & K & w_K& p & K & w_K & p>5 & K & w_K\\
\hline
\multirow{2}{*}{2} & \qq(\sqrt{2}, \sqrt{-1}) & 4 &\multirow{3}{*}{5} &
\qq(\sqrt{5}, \sqrt{-1}) & 2 & \multirow{2}{*}{$p\equiv 1\; (4)$}&
\qq(\sqrt{p},\sqrt{-1}) & 2\\
                   & \qq(\sqrt{2}, \sqrt{-3}) & 3& & \qq(\sqrt{5},
                   \sqrt{-3}) & 3 & &\qq(\sqrt{p},\sqrt{-3})& 3 \\ 
\cline{1-3} \cline{7-9} 
 \multirow{2}{*}{3} & \qq(\sqrt{3}, \sqrt{-1}) & 12& &\qq(\zeta_{10})
 & 5& \multirow{3}{*}{$p \equiv 3 \; (4)$} & \qq(\sqrt{p},\sqrt{-1}) & 4\\
     & \qq(\sqrt{3}, \sqrt{-2}) & 2& & & & & \qq(\sqrt{p}, \sqrt{-2})
     & 2\\
& & & & & & & \qq(\sqrt{p}, \sqrt{-3}) & 3\\
   \hline
\end{tabular}

\bigskip

It is well known that the class numbers (cf. \cite[Theorem
11.1]{Washington-cyclotomic})
\begin{equation}
  \label{xeq:48}
  h(\qq(\zeta_8))=h(\qq(\zeta_{10}))=h(\qq(\zeta_{12}))=1.
\end{equation}
Using Magma \cite{MR1484478}, one easily calculates that 
\begin{gather}
  h(\qq(\sqrt{2}, \sqrt{-3} ))=h(\qq(\sqrt{5},
  \sqrt{-1} ))=h(\qq(\sqrt{5}, \sqrt{-3} ))=1, \label{xeq:49}\\ 
  h(\qq(\sqrt{3}, \sqrt{-2} ))=2. \label{xeq:50}
\end{gather}
\end{sect}



\begin{sect}\label{subsec:disc-criterion-for-jK}
  Let $E_j=\qq(\sqrt{-j})$ for $j=1,2,3$, and $\d_{E_j}$ be the
  discriminant of $E_j$. Suppose that $p$ is \textit{odd}, and $\d_F$
  is the discriminant of $F=\qq(\sqrt{p})$. Consider the biquadratic
  field $K_j:=\qq(\sqrt{p}, \sqrt{-j})$, which is the compositum
  of $F$ with $E_j$. If $p=3$, we only take
  $K_1$ and $K_2$.  Proposition~\ref{prop:criterion-for-jK}
  shows the following simple but mysterious criterion:
  \begin{equation}
    \label{xeq:16}
 Q_{K_j/F}=1 \qquad \Longleftrightarrow\qquad \gcd(\d_F, \d_{E_j})=1.    
  \end{equation}
\end{sect}

\begin{sect}\label{subsec:Herglotz-formula}
  Suppose for the moment that $F=\qq(\sqrt{d})$ is an arbitrary real
  quadratic field, and $K$ is the compositum of $F$ with an imaginary
  quadratic field $E$. By the work of Herglotz \cite{MR1544516}, if
  $K\neq \qq(\sqrt{2}, \sqrt{-1})$, then
  \begin{equation}
    \label{eq:64}
    h(K)=Q_{K/F}h(F)h(E)h(E')/2,
  \end{equation}
  where $E'$ is the only other imaginary quadratic subfield of $K$
  distinct from $E$. In particular, if $F=\qq(\sqrt{p})$,
  $K_j=\qq(\sqrt{p}, \sqrt{-j})$ and $\Bbbk_j=\qq(\sqrt{-pj})$ with
  $j=1,2,3$ and $p\geq 5$, then
\begin{equation}
  \label{eq:70}
  h(K_j)=
  \begin{cases}
    h(F)h(\Bbbk_j) &\qquad \text{if } j=1,2 \text{
      and } p\equiv 3 \pmod{4};\\
    h(F)h(\Bbbk_j)/2 &\qquad \text{otherwise}.
  \end{cases}
\end{equation}
Here we used the facts that $h(\qq(\sqrt{-j}))=1$ for all $j\in \{1,2,3\}$ and
$Q_{K_j/F}$ is calculated in
Proposition~\ref{prop:criterion-for-jK}.
\end{sect}


\begin{sect}\label{subsec:gp-structure-of-unit-mod-2-for-K}
  Suppose that $p$ is odd, and $K=K_1=\qq(\sqrt{p}, \sqrt{-1})$.  Let
  $L=\qq(\sqrt{p^*})\subset K$, where $p^*:=\left(\frac{-1}{p}\right)
  p$, and $\left(\frac{\cdot}{p}\right)$ is the Legendre symbol.  Then
  $O_L=\zz\oplus \zz\omega_p$, with $\omega_p:=(1+\sqrt{p^*})/2\in
  O_L$. Since $\gcd(\d_L, \d_{\qq(\sqrt{-1})})=1$,
  we have $O_K=O_L[\sqrt{-1}]$ and a $\zz$-basis of $O_K$ is given
  by
  \begin{equation}\label{xeq:40}
    \left\{1,\quad \frac{1+\sqrt{p^*}}{2}, \quad \sqrt{-1}, \quad
      \frac{\sqrt{-1}+\sqrt{-p^*}}{2}\right\}.
  \end{equation}
  We claim that $\abs{(O_K/2O_K)^\times}=
  4\left(2-\left(\frac{2}{p}\right)\right)$.  Indeed, we have
\begin{equation}
  \label{eq:52}
  O_K/2O_K \cong (O_L/2O_L)[t]/(t^2+1) = (O_L/2O_L)[t]/((t+1)^2),
  \end{equation}
  with the isomorphism sending $\sqrt{-1}\mapsto \bar{t}$, which denotes
  the image of $t$ in the quotient.  The isomorphism (\ref{eq:52})
  gives rise to an exact sequence
\begin{equation}
  \label{eq:53}
   0 \to (O_L/2O_L) \to (O_K/2O_K)^\times \to (O_L/2O_L)^\times\to 1. 
\end{equation}
Note that $2$ is unramified in $L$, and
\begin{equation}
  \label{eq:54}
  O_L/2O_L \simeq
  \begin{cases}
    \ff_2\oplus \ff_2  &\qquad  \text{if } \left(\frac{2}{p}\right)=1; \\
\ff_4 &\qquad  \text{if } \left(\frac{2}{p}\right)=-1. 
  \end{cases}
\end{equation}
Hence the exact sequence (\ref{eq:53}) splits. More precisely, 
  \begin{equation}\label{xeq:27}
    (O_K/2O_K)^\times \simeq
    \begin{cases}
      (\zmod{2})^2 &\qquad 
  \text{ if } \left(\frac{2}{p}\right)=1;\\
  (\zmod{3})\oplus
    (\zmod{2})^2     &\qquad    
  \text{ if } \left(\frac{2}{p}\right)=-1.
    \end{cases}
  \end{equation}




\end{sect}

\begin{sect}\label{subsec:orders-in-K1-containing-B4}
  Consider the order $B_{1,4}:=\zz[\sqrt{p},
  \sqrt{-1}]=\zz[\sqrt{p^*}, \sqrt{-1}]$ in $K=\qq(\sqrt{p},
  \sqrt{-1})$ with $p$ odd. Since $\zz[\sqrt{p^*}]/2O_L\cong \ff_2$, 
  we
  have $2O_K\subset B_{1,4}$, and
  \begin{equation}
    \label{eq:61}
  O_K/2O_K\supset B_{1,4}/2O_K\cong
  (\zz[\sqrt{p^*}]/2O_L)[t]/((t+1)^2)\cong \ff_2[t]/((t+1)^2)   
  \end{equation}
under
  the isomorphism (\ref{eq:52}). 
  In particular, $(B_{1,4}/2O_K)^\times
  \cong \zmod{2}$.  

  Note that $O_L/2O_L$ is spanned by the image of $1$ and $\omega_p$
  over $\ff_2$. One easily checks that the only other ring
  intermediate to
  \begin{equation}
    \label{eq:62}
\ff_2[t]/((t+1)^2) \subset (O_L/2O_L)[t]/((t+1)^2)=
  (O_L/2O_L)\oplus (O_L/2O_L)(1+\bar{t}\,) 
  \end{equation}
is $\ff_2\oplus (O_L/2O_L)(1+\bar{t}\,)$.  It follows that
$B_{1,2}:=\zz+\zz\sqrt{p}+\zz\sqrt{-1}+\zz y_p^*$ is the only
nontrivial suborder intermediate to $B_{1,4}\subset O_K$, where
\[y_p^*:=\omega_p(1+\sqrt{-1})= (1+\sqrt{p^*})(1+\sqrt{-1})/2. \]
However, it is more convenient to define
$y_p:=(1+\sqrt{-1})(1+\sqrt{p})/2$, then
$B_{1,2}=\zz+\zz\sqrt{p}+\zz\sqrt{-1}+\zz y_p$ as well. Note that
$y_p^2=(1+p)\sqrt{-1}/2+\sqrt{-p}$, so $B_{1,2}=\zz[\sqrt{-1},
y_p]$. Since $B_{1,2}/2O_K\cong \ff_2\oplus (O_L/2O_L)(1+\bar{t}\,)$,
we have
 \[(B_{1,2}/2O_K)^\times\cong O_L/2O_L\simeq (\zmod{2})^2. \]
\end{sect}




\section{$O_F$-orders in $K$} 
\label{sec:OF-orders}
We keep the notations of Section~\ref{sec:units-totally-imag}. In
particular, $F=\qq(\sqrt{p})$ and its ring of integers is denoted by
$O_F$. We will classify all the quadratic $O_F$-orders $B$ satisfying
the following two conditions: 
\begin{enumerate}[(i)]
\item the fraction field of $B$ is a totally imaginary quadratic
  extension $K$ of $F$;
\item $w(B)=[B^\times:O_F^\times]>1$.
\end{enumerate}
Unless specified otherwise, the notation $B$ will be reserved for such
orders throughout this section. 
The quotient group $B^\times/O_F^\times$ is a subgroup of the finite
cyclic group $O_K^\times/O_F^\times$, hence $w(B)$ divides
$w_K=[O_K^\times:O_F^\times]$. Therefore, $K$ must be one of the
fields given in the table of
Section~\ref{subsec:quotient-units-gp-cyclic}.


\begin{prop}\label{prop:wK-prime-order-maximal}
  Suppose that $w_K$ is a prime. Then $B=O_K$ is the unique
  $O_F$-order in $K$ such that $w(B)>1$.
\end{prop}
\begin{proof}
  By the table of Section~\ref{subsec:quotient-units-gp-cyclic},
  $w_K$ is a prime only when $w_K=2, 3, 5$.  Then $O_K^\times/O_F^\times$ is a
  cyclic group of prime order with a nontrivial subgroup
  $B^\times/O_F^\times$. Therefore, $B^\times/O_F^\times=O_K^\times/O_F^\times$, so $B^\times
  =O_K^\times$.  Then $B\supseteq O_F[u]$ for any $u\in O_K^\times$.

 If $w_K=5$, then $F=\qq(\sqrt{5})$ and $K=\qq(\zeta_{10})$. We have
 $B\supseteq O_F[\zeta_{10}]\supseteq \zz[\zeta_{10}]$. But
 $\zz[\zeta_{10}]$ is the maximal order in $K$. So
 $B=O_K=\zz[\zeta_{10}]$. 

 If $Q_{K/F}=2$ and $w_K=2$, then $p\equiv 3 \pmod{4}$ and
 $K=F(\sqrt{-\epsilon})=\qq(\sqrt{p},
 \sqrt{-2})$. Proposition~\ref{prop:ring-of-int-in-quad-ext} shows
 that $O_F[\sqrt{-\epsilon}]=O_K$ is the maximal order in $K$. So
 $B=O_K=O_F[\sqrt{-\epsilon}]$.

 Suppose that $Q_{K/F}=1$, $p$ is odd and $K\neq \qq(\zeta_{10})$. In
 other words, we assume one of the following holds: 
 \begin{itemize}
 \item $p\equiv 1 \pmod{4}$, and $K\neq \qq(\zeta_{10})$;
 \item $p\equiv 3 \pmod{4}$, $p\neq 3$, and $K=F(\zeta_6)=\qq(\sqrt{p},
   \sqrt{-3})$.
 \end{itemize}
 Then we have $K=\qq(\sqrt{p}, \sqrt{-j})$ with $ j\in\{1, 3\}$, which
 depends
 on $p$.  By Section~\ref{subsec:disc-criterion-for-jK}, the
 assumption $Q_{K/F}=1$ guarantees that the discriminants of
 $\qq(\sqrt{p})$ and $\qq(\sqrt{-j})$ are relatively prime.  Let
 $\zeta=\zeta_4$ if $j=1$ and $\zeta=\zeta_6$ if $j=3$. Then
 $B\supseteq O_F[\zeta]$. By \cite[Proposition
 III.17]{Lang-ANT}, $O_F[\zeta]$ is the maximal order in $K$. Therefore
 $B=O_K$.

 The only remaining case to consider is $F=\qq(\sqrt{2})$ and
 $K=F(\zeta_6)=\qq(\sqrt{2},\sqrt{-3})$. We note that the
 discriminants of $\qq(\sqrt{2})$ and $\qq(\sqrt{-3})$ are again
 relatively prime. So the same argument as above shows that
 $B=O_K$. 
\end{proof}
\begin{lem}\label{lem:wB-is-2-or-4}
  Suppose that $p\equiv 3 \pmod{4}$ and $K=\qq(\sqrt{p},
  \sqrt{-1})$. Let $B\subseteq O_K$ be a quadratic $O_F$-order
  with $2\mid w(B)$. Then $B_{1,4}=\zz[\sqrt{p},\sqrt{-1}]\subseteq
  B$. Moreover, $4\mid w(B)$ if and only if
  $y_p=(1+\sqrt{-1})(1+\sqrt{p})/2\in B$.
\end{lem}
\begin{proof}
  If $p=3$, then $O_K^\times/O_F^\times$ is a cyclic 
  group of order $12$, generated by
  the image of $z=\sqrt{\epsilon \zeta_{12}}\in O_K^\times$.  
  Since $2\mid 
  w(B)$, we have $B\ni z^6=\epsilon^3\sqrt{-1}$. 
  Then $\sqrt{-1}\in B^\times$
  as $\epsilon \in O_F^\times\subset B^\times$. We have $4\mid w(B)$ if and
  only if $B\ni z^3 =\epsilon\sqrt{\epsilon}\zeta_{8}$, or
  equivalently, $B\ni \sqrt{\epsilon}\zeta_{8}$. 

  If $p>3$ and $p\equiv 3 \pmod{4}$, then $O_K^\times/O_F^\times$ 
  is a cyclic group of
  order $4$ generated by $z=\sqrt{\epsilon \zeta_4}$. If $2\mid w(B)$,
  then $B\ni z^2= \epsilon \sqrt{-1}$, so $\sqrt{-1}\in B$.  Moreover,
  $w(B)=4$ if and only if $B\ni z=\sqrt{\epsilon}\zeta_8$. 

  It remains to show that $\sqrt{\epsilon}\zeta_8\in B$ if and only if
  $y_p\in B$.  By Proposition~\ref{prop:sqrt-epsilon-and-sqrt-2},
  there exists $m,n\in \zz$ such that $\sqrt{\epsilon/2}=
  m+n\sqrt{p}+(1+\sqrt{p})/2$. We then have
\[\sqrt{\epsilon}\zeta_8=\sqrt{\epsilon/2}\cdot
(\sqrt{2}\zeta_8)=\left(m+n\sqrt{p}+\frac{1+\sqrt{p}}{2}\right)(1+\sqrt{-1}).\]
But $B$ already contains $\zz[\sqrt{p}, \sqrt{-1}]$ by the above
arguments, so $\sqrt{\epsilon}\zeta_8\in B$ if and only if
$y_p=(1+\sqrt{-1})(1+\sqrt{p})/2\in B$.
\end{proof}

\begin{prop}\label{prop:list-of-order-p-cong-3-mod-4}
  Suppose that $p\equiv 3\pmod{4}$ and $K=\qq(\sqrt{p}, \sqrt{-1})$. The
  $O_F$-orders $B\subseteq O_K$ with $2\mid w(B)$ are:
  \begin{alignat*}{2}
    &O_K,   &\qquad  w(O_K)&=4\gcd(p,3); \\
    B_{1,2}&=\zz[\sqrt{-1}, y_p],
    &\qquad  w(B_{1,2})&=4;\\
    B_{1,4}&=\zz[\sqrt{p}, \sqrt{-1}], &\qquad w(B_{1,4})&=2.
  \end{alignat*}
  If $p> 3$, the above is a complete list of $O_F$-orders in $K$
  with $w(B)>1$. If $p=3$, there is an extra order $B_{1,3}=\zz[\sqrt{3},
  \zeta_6]$ with $w(B_{1,3})=3$.
\end{prop}
\begin{proof}
  Recall that $w_K=4$ or $12$. Given any $B\subseteq O_K$ with
  $w(B)>1$, we have either $2\mid w(B)$ or $w(B)=3$, with the latter
  case possible only if $p=3$.

  Suppose that $2\mid w(B)$. Then $B\supseteq B_{1,4}:=\zz[\sqrt{p},
  \sqrt{-1}]$ by Lemma~\ref{lem:wB-is-2-or-4}. By
  Section~\ref{subsec:orders-in-K1-containing-B4}, $B_{1,2}$ is the
  only $O_F$-order of index $2$ intermediate to $B_{1,4}\subset O_K$.
  Since $y_p\not\in B_{1,4}$, we have $w(B_{1,4})=2$ by
  Lemma~\ref{lem:wB-is-2-or-4}.  On the other hand, $4\mid
  w(B_{1,2})$. So $w(B_{1,2})=4$ if $p>3$. Note that
  $\zeta_{12}=(\sqrt{3}+\sqrt{-1})/2\not\in B_{1,2}$ if $p=3$. Hence
  $w(B_{1,2})=4$ in this case as well.

  Suppose that $p=3$, $z=\sqrt{\epsilon\zeta_{12}}$ and $3\mid
  w(B)$. Then $B\ni z^4=\epsilon^2\zeta_6$ and hence $B\supseteq
  \zz[\sqrt{3}, \zeta_6]$. 
   A $\zz$-basis of $B_{1,3}:=\zz[\sqrt{3}, \zeta_6]$
  is given by 
\[\left\{ 1,\quad  \sqrt{3}, \quad \zeta_6=\frac{1+\sqrt{-3}}{2},
  \quad \sqrt{3}\zeta_6=\frac{\sqrt{3}+3\sqrt{-1}}{2} \right\}.\] 
One
easily checks that $[O_K:B_{1,3}]=3$. Hence the only other $O_F$-order
containing $B_{1,3}$ is $O_K$ itself. Since $\sqrt{-1}\not\in B_{1,3}$, we
have $w(B_{1,3})=3$. 
\end{proof}

For the rest of this section, we study the class numbers $h(B)$ of
those non-maximal orders $B$ with $w(B)>1$.

\begin{sect}
  For the moment let us assume that $K$ is an arbitrary number field,
  and $B\subseteq O_K$ is an order in $K$ with conductor $\f$.  The
  class number of $B$ is given by \cite[Theorem I.12.12]{MR1697859}
  \begin{equation}
    \label{xeq:36}
    h(B)=\frac{h(O_K) [{(O_K/\f)}^\times:
      {(B/\f)}^\times]}{[O_K^\times :B^\times]}\,. 
  \end{equation}
  We leave it as an exercise to show that $[(O_K/\a)^\times:
  (B/\a)^\times]= [(O_K/\f)^\times: (B/\f)^\times]$ for any nonzero
  ideal $\a$ of $O_K$ contained in $\f$. Therefore,
\begin{equation}
  \label{xeq:37}
    h(B)=\frac{h(O_K) [{(O_K/\a)}^\times:
      {(B/\a)}^\times]}{[O_K^\times :B^\times]}\,. 
  \end{equation}
\end{sect}

\begin{lem}\label{lem:class-number-b12-b14}
  Suppose that $p\equiv 3\pmod{4}$ and $K=\qq(\sqrt{p},
  \sqrt{-1})$. Let $B_{1,2}$ and $B_{1,4}$ be the orders in
  Proposition~\ref{prop:list-of-order-p-cong-3-mod-4}. We have 
\begin{equation}
  \label{xeq:31}
h(B_{1,2})=h(B_{1,4})=\left(2-\left(\frac{2}{p}\right)\right)h(O_K)  
\end{equation}
if $p>3$ and $p\equiv 3\pmod{4}$. If $p=3$,
then $h(B_{1,2})=h(B_{1,4})=h(O_K)$. 
\end{lem}
\begin{proof}
  By Section~\ref{subsec:orders-in-K1-containing-B4}, we have $
  O_K\supset B_{1,2}\supset B_{1,4} \supset 2O_K$. So take $\a=2O_K$
  in (\ref{xeq:37}). It has been shown in
  Sections~\ref{subsec:gp-structure-of-unit-mod-2-for-K} and
  \ref{subsec:orders-in-K1-containing-B4} that
\[\abs{(O_K/2O_K)^\times}=4\left(2-\left(\frac{2}{p}\right)\right),
\quad \abs{(B_{1,2}/2O_K)^\times}=4 \quad \text{and} \quad
\abs{(B_{1,4}/2O_K)^\times}=2.\] 
On the other hand, $[O_K^\times:
B^\times]=w_K/w(B)$ for $B=B_{1,2}$ or $B_{1,4}$. Recall that $w_K=4$
if $p>3$ and $w_K=12$ if $p=3$. The lemma now follows from
Proposition~\ref{prop:list-of-order-p-cong-3-mod-4}, where it has been
shown that $w(B_{1,2})=4$ and $w(B_{1,4})=2$. 
\end{proof}



\begin{sect}\label{subsec:p-2-B2}
  Assume that $F=\qq(\sqrt{2})$ and
  $K=F(\zeta_8)=\qq(\sqrt{2},\sqrt{-1})$. Then $w_K=4$, and
  $O_K^\times/O_F^\times\cong \zmod{4}$.  Any $B\subseteq O_K$ with
  $w(B)>1$ must contain $O_F[\zeta_8^2]=\zz[\sqrt{2},\sqrt{-1}]$. By
  Exercise 42(b) of \cite[Chapter 2]{MR0457396}, a $\zz$-basis of
  $O_K$ is given by $\left\{1,\quad \sqrt{-1}, \quad \sqrt{2}, \quad
    (\sqrt{2}+\sqrt{-2})/2\right\}$. Let $B=\zz[\sqrt{2}, \sqrt{-1}]$,
  which is a sublattice of $O_K$ of index $2$. Therefore, there are no
  other quadratic $O_F$-orders $B'$ in $K$ with $w(B')>1$ and $B'\neq
  O_K$. We have
  \begin{equation}
    \label{xeq:17}
    w(O_K)=4 \quad \text{and} \quad w(B)=2.
  \end{equation}
  Note that $\sqrt{2}O_K\subseteq B$.  The ideal $\p=(1+\zeta_8)O_K$
  is the unique prime ideal above $2$.  Therefore, $O_K/\sqrt{2}O_K$
  is a two-dimensional $\ff_2$-algebra whose unit group
  $(O_K/\sqrt{2}O_K)^\times=(O_K/\p^2)^\times \cong \zmod{2}$.  
  Since $[O_K:B]=2$,
  we have $B/\sqrt{2}O_K\cong \ff_2 $. It follows that
   $ h(B)=h(O_K)=1.$
\end{sect}

\begin{sect}\label{sec:3.7}
  Let $K=\qq(\sqrt{3}, \sqrt{-1})$ and $B_{1,3}=\zz[\sqrt{3},
  \zeta_6]$.  We have $\sqrt{-3}O_K\subset B_{1,3}$. 
  On the other hand, $\sqrt{-3}O_K$ is a prime ideal in $O_K$ 
  with residue field $\ff_9$. 
  Since $[O_K: B_{1,3}]=3$, we have $B_{1,3}/\sqrt{3}O_K\cong
  \ff_3$. Therefore, $h(B_{1,3})=h(O_K)=1$.
\end{sect}

\begin{sect}\label{sec:hO}
  Let $D$ be a totally definite quaternion algebra over
  $F=\qq(\sqrt{p})$ of discriminant ideal $\calD\subset O_F$, and $\calO$ 
  an Eichler order of level $\calN$, where $\calN\subset O_F$ is a
  square-free prime-to-$\calD$ ideal. The mass formula 
 \cite[Chapter V, Corollary 2.3]{vigneras} states that
\begin{equation}
  \label{eq:mass_formula}
  \Mass(\calO)=\frac{1}{2} \zeta_F(-1) h(F) \prod_{\grp
  | \calD } 
  (N(\grp)-1) \prod_{\grp|\calN} (N(\grp)+1)=:M,    
\end{equation}
where $\zeta_F(s)$ is the Dedekind zeta function of $F$. For any
$O_F$-order $B$ in a quadratic extension $K/F$, we define the Artin
symbol 
\[ \left(\frac{K}{\grp}\right ):=
\begin{cases}
  \ 1 & \text{if $\grp$ splits in $K$;} \\
  -1 & \text{if $\grp$ is inert in $K$;} \\
  \ 0 & \text{if $\grp$ is ramified in $K$;} \\
\end{cases} \]
and the Eichler symbol
\[ \left(\frac{B}{\grp}\right ):=
\begin{cases}
  \! \left(\frac{K}{\grp}\right ) & \text{if $\grp\nmid \grf(B)$;} \\
  \ 1 & \text{otherwise;}  
\end{cases} \]
where $\grf(B)\subseteq O_F$ is the conductor of $B$. Define
\begin{equation}\label{eq:mBKN}
  \begin{split}
 & E_{K,\calD, \calN}:=\prod_{\grp|\calD} \left (
  1-\left(\frac{K}{\grp}\right )\right ) \prod_{\grp|\calN} \left (
  1+\left(\frac{K}{\grp}\right )\right ), \quad \text{and} \\
 & E_{B,\calD, \calN}:=\prod_{\grp|\calD} \left (
  1-\left(\frac{B}{\grp}\right )\right ) \prod_{\grp|\calN} \left (
  1+\left(\frac{B}{\grp}\right )\right ).     
  \end{split}
\end{equation}
By the formula \cite[p.~94]{vigneras}, one has
\[ \prod_{\grp} m_\grp(B)=E_{B,\calD,\,\calN}. \] 
For an ideal $\gra\subset O_F$ and a square-free integer $n$, we can
write $\gra=\gra_{(n)} \gra^{(n)}$ as the product of a 
$n$-primary ideal $\gra_{(n)}$ and a prime-to-$n$ ideal $\gra^{(n)}$.
For any two $O_F$-ideals $\gra, \grb$, we set
\[ C_{\gra, \grb}:=\delta_{\gra,(1)} 2^s, \]
where $\delta_{\gra,(1)}$ is the usual delta function and 
$s$ is the number of prime ideals $\grp$ dividing $\grb$. If there is
a unique prime ideal $\grp_2$ of $O_F$ lying over $2$ and the
conductor $\grf(B)$ is $\grp_2$-primary, then
\begin{equation}
  \label{eq:c}
  \begin{split}
  E_{B,\calD,\,\calN}& =E_{B,\calD_{(2)},
  \,\calN_{(2)}}\cdot E_{B,\calD^{(2)},
  \,\calN^{(2)}}=C_{\calD_{(2)},\,\calN_{(2)}}\cdot  E_{K,\!\calD^{(2)},\,
  \calN^{(2)}}.    
  \end{split}
\end{equation}    

We now have everything to compute the class number $h(\calO)$. 
Recall that $K_j=\qq(\sqrt{p}, \sqrt{-j})$ for $j\in\{1,2,3\}$.   
By Section~\ref{subsec:quotient-units-gp-cyclic} and
Proposition~\ref{prop:wK-prime-order-maximal}, if $p\equiv 1 \pmod{4}$
and $p>5$, then the only orders with nonzero 
contributions to the elliptic
part $\Ell(\calO)$ are $O_{K_1}$ and $O_{K_3}$, with
$w(O_{K_1})=2$ and $w(O_{K_3})=3$ respectively.  We have
\begin{equation}
  \label{eq:37}
  h(\calO)=M+\frac{1}{4}h(K_1)E_{K_1,\calD,\,\calN}+\frac{1}{3}h(K_3)E_{K_3,\calD,\,\calN}   
\end{equation}
for $p\equiv 1 \pmod{4}$ and $p>5$.
On the other hand, for  $p \equiv 3 \pmod{4}$ and $p> 5$, we have
calculated the following numerical invariants of all orders $B$ with
$w(B)>1$ 
(see Section~\ref{subsec:quotient-units-gp-cyclic}, 
Propositions~\ref{prop:wK-prime-order-maximal} 
and \ref{prop:list-of-order-p-cong-3-mod-4} and 
Lemma~\ref{lem:class-number-b12-b14}):\\

\renewcommand{\arraystretch}{1.5}
\noindent\begin{tabular}{|>{$}c<{$}||*{5}{>{$}c<{$}|}}
\hline
  p\equiv 3 \pmod{4}
  & O_{K_1} & B_{1,2} & B_{1,4} & O_{K_2} & O_{K_3}\\
\hline
h(B) & h(K_1)& \left(2-\left(\frac{2}{p}\right)\right)h(K_1)
&\left(2-\left(\frac{2}{p}\right)\right)h(K_1) & h(K_2) & h(K_3)\\
w(B) & 4 & 4 & 2 & 2 & 3\\
\hline
\end{tabular} \\

Therefore, by Eichler's class number formula we obtain
\begin{equation}
  \label{eq:36}
  \begin{split}
  h(\calO)=& M +
  \frac{5}{8}\left(2-\left(\frac{2}{p}\right)\right) h(K_1)\, 
  C_{\calD_{(2)},\,\calN_{(2)}}\,  
  E_{K_1,\!\calD^{(2)},\, \calN^{(2)}}+ \\
& \frac{3}{8}h(K_1)\,E_{K_1,\calD,\,\calN}+
\frac{1}{4}h(K_2)\,E_{K_2,\calD,\,\calN}
+\frac{1}{3}h(K_3)\,E_{K_3,\calD,\,\calN}    
  \end{split}
\end{equation}
for $p \equiv 3 \pmod{4}$ and $p> 5$.
For $p=2,3,5$, the formulas for $h(\calO)$ can be obtained in the same
way using Sections~\ref{subsec:quotient-units-gp-cyclic}, 
\ref{subsec:p-2-B2} and \ref{sec:3.7}. 
\end{sect}



\section{Quadratic proper  $\zz[\sqrt{p} ]$-orders in $K$}
\label{sec:quadr-prop-zzsqrtp-orders}
Throughout this section, we assume that $p\equiv 1\pmod{4}$ and let
$A=\zz[\sqrt{p} ]$. It is an order of index $2$ in
$O_F=\zz+\zz(1+\sqrt{p})/2$ with $A/2O_F\cong \ff_2$. 
We will classify all the quadratic proper $A$-orders $B$ satisfying
the following two conditions: 
\begin{enumerate}[(i)]
\item the fraction field of $B$ is a totally imaginary quadratic
  extension $K$ of $F$;
\item $w(B):=[B^\times:A^\times]>1$.
\end{enumerate}
First we need some knowledge about the group $A^\times$.



\begin{lem}\label{lem:p-equiv-1mod8-epsilon}
  If $p\equiv 1\pmod{8}$, then $A^\times= O_F^\times$. In particular,
  the fundamental unit $\epsilon\in A^\times$. 
\end{lem}
\begin{proof}
By our assumption on $p$, $2O_F=\p_1\p_2$, where $\p_1$ and $\p_2$
are maximal ideals of $O_F$ with residue fields
$O_F/\p_1=O_F/\p_2=\ff_2$. Therefore, 
\[(O_F/2O_F)^\times \cong (O_F/\p_1)^\times\times
(O_F/\p_2)^\times\]
is a trivial group. We have $u\equiv 1\pmod{2O_F}$ for any $u\in
O_F^\times$. Hence $u\in A\cap O_F^\times=A^\times$. 
\end{proof}

\begin{sect}\label{subsec:fund-unit-in-A-or-not}
  If $p\equiv 5\pmod{8}$, $2$ is inert in $O_F$, and we have
  $(O_F/2O_F)^\times \simeq \ff_4^\times \simeq \zmod{3}$. Let
  $U^{(1)}$ be the kernel of the map $O_F^\times\to
  (O_F/2O_F)^\times$. Since $(A/2O_F)^\times $ is the trivial
  subgroup of $(O_F/2O_F)^\times$, we have $A^\times =U^{(1)}$.
  If $\epsilon\in A$, then $O_F^\times=A^\times=U^{(1)}$; otherwise,
  $O_F^\times/A^\times \simeq \zmod{3}$, and $O_F^\times\to (O_F/2O_F)^\times$ is
  surjective.  Here we are in a more complicated situation since both
  cases may occur, and whether $\epsilon\in A^\times$ or not can no
  longer be determined by a simple congruence condition on $p$. The list of
  $p\equiv 5 \pmod{8}$ and $p<1000$ such that $\epsilon \in A^\times$
  are given bellow:
  \[37, \, 101,\, 197, \,269, \,349, \,373,\, 389, \,557, \,677,
  \,701, \,709, \,757,\, 829,\, 877, \,997.\] This is the sequence
  A130229 in the OEIS \cite{oeisA130229}. For any $p\equiv 1
  \pmod{4}$, we define
\begin{equation}
  \label{eq:4}
  \varpi:=[O_F^\times:A^\times]\in \{ 1, 3\}. 
\end{equation}
By Lemma~\ref{lem:p-equiv-1mod8-epsilon}, $\varpi=1$ if $p\equiv
1\pmod{8}$. 
\end{sect}

\begin{sect}
  Let $A_+^\times\subset A^\times$
  be the subgroup consisting of all the totally positive elements of
  $A^\times$.  We claim that
  \begin{equation}\label{xeq:28}
A_+^\times= (A^\times)^2.     
  \end{equation}
  If $\epsilon\in A$, then $A^\times= O_F^\times=\dangle{\epsilon}\times
  \{\pm 1\}$. Since $\epsilon$ is not totally positive by
  Lemma~\ref{lem:norm}, we have
  $A_+^\times=\dangle{\epsilon^2}=(A^\times)^2$. If $\epsilon\not\in
  A$, then $A^\times=\dangle{\epsilon^3}\times \{\pm 1\}$ by
  Section~\ref{subsec:fund-unit-in-A-or-not}. It follows that
  $A_+^\times=\dangle{\epsilon^6}=(A^\times)^2$. So either way,
  (\ref{xeq:28}) holds. 
\end{sect}

\begin{lem}\label{lem:possible-fields-K1-K3}
  Let $K$ be a totally imaginary quadratic extension of $F$ such that
  there exists a quadratic proper  $A$-order $B\subset K$ with
  $w(B)>1$.  Then $K$ is necessarily one of the following
  \[K_1=\qq(\sqrt{p}, \sqrt{-1}), \qquad K_3=\qq(\sqrt{p},
  \sqrt{-3}).\]  
Moreover, if $K=K_1$, then $B\supseteq \zz[\sqrt{p},\sqrt{-1}]$. 
\end{lem}
\begin{proof}
 By Section~\ref{subsec:list-of-roots-of-unity}, it
  is enough to show that $\bmu_K\neq \{\pm 1\}$, and $K\neq
  \qq(\zeta_{10})$ if $p=5$.

  First, if $p=5$, the fundamental unit
  $\epsilon=(1+\sqrt{5})/2\not\in A$, and by
  Section~\ref{subsec:fund-unit-in-A-or-not}, $O_F^\times/A^\times\cong
  \zmod{3}$.  Assume $K=\qq(\zeta_{10})$, then
  \[\{1\}\varsubsetneq B^\times/A^\times \subseteq
  O_K^\times/A^\times=\dangle{\bar{\epsilon}}\oplus
  \dangle{\bar{\zeta}_{10}}\cong \zmod{3}\oplus \zmod{5}, \] 
  where
  $\bar{\epsilon}$ and $\bar{\zeta}_{10}$ denote the image of
  $\epsilon$ and $\zeta_{10}$ respectively in the quotient
  $O_K^\times/A^\times$.  Note that $B^\times/A^\times$ can not
  contain the subgroup $\dangle{\bar{\epsilon}}\cong \zmod{3}$. Otherwise,
  $B\ni \epsilon$, which implies that $B\supset \zz[\epsilon]=O_F$,
  contradicting the assumption that $B$ is a proper $A$-order.  On
  the other hand, if $B^\times/A^\times \supseteq
  \dangle{\bar{\zeta}_{10}}\cong \zmod{5}$, then $B\ni \zeta_{10}$. Hence
  $B\supseteq \zz[\zeta_{10}]$, which is the maximal order in
  $K=\qq(\zeta_{10})$. Again this leads to a contradiction to the
  assumption on $B$. We conclude that $K\neq \qq(\zeta_{10})$ if
  $p=5$.

  Recall that  $\bmu_K\supseteq \phi_K(B^\times)$, where
  $\phi_K: u\mapsto u/\iota(u)$ is the map given in
  (\ref{xeq:51}). Clearly, $\phi_K(B^\times)\neq \{1\}$. Otherwise,
  $B^\times \subseteq O_F^\times\cap B=A^\times$, contradicting the
  assumption that $w(B)>1$.  

  Suppose that $-1=\phi_K(u)$ for some $u\in B^\times$. We have
  $-u^2=\Nm_{K/F}(u)\in A_{+}^\times$, the group of totally positive
  units of $A$. Since $A_+^\times= (A^\times)^2$ by (\ref{xeq:28}),
  multiplying $u$ by a suitable element of $A^\times$, we may assume
  that $u^2=-1$. Therefore, $K=K_1=F(\sqrt{-1})$. On the other hand,
  if $K=K_1$, then by Section~\ref{subsec:intro-of-j-K-over-F},
  $\phi_K(O_K^\times)=\bmu_K^2=\{\pm 1\}$ since $Q_{K/F}=1$.  Therefore,
  $\phi_K(u)=-1$ for all $u\in B^\times-A^\times$. We have in fact
  shown that $B\ni \sqrt{-1}$ for all proper $A$-orders in $K_1$ with
  $w(B)>1$.

  Lastly, if $-1\not\in \phi_K(B^\times)$, then $\phi_K(B^\times)$
  contains a root of unity which is not in $F$. In particular,
  $\bmu_K\neq \{\pm 1\}$ and $w_K> 1$. By
  Section~\ref{subsec:list-of-roots-of-unity}, we must have
  $K=K_3=F(\sqrt{-3})$ since all other possibilities have been
  exhausted.
\end{proof}

\begin{sect}
  Suppose that $K=K_1$. It has been shown in
  Lemma~\ref{lem:possible-fields-K1-K3} that $B\supseteq
  B_{1,4}=\zz[\sqrt{p}, \sqrt{-1} ]$. By
  Section~\ref{subsec:orders-in-K1-containing-B4}, \[B_{1,2}=
  \zz+\zz\sqrt{p}+\zz\sqrt{-1}+\zz(1+\sqrt{-1})(1+\sqrt{p})/2 \] is
  the only other proper $A$-order that contains $B_{1,4}$.  The class
  numbers of $B_{1,2}$ and $B_{1,4}$ can be calculated exactly in the
  same way as in Lemma~\ref{lem:class-number-b12-b14}.  Let $B$ be
  either $B_{1,2}$ or $B_{1,4}$.  If $\epsilon\in A$, then
  $O_K^\times/A^\times=O_K^\times/O_F^\times \cong \zmod{2}$. Hence
  $B^\times=O_K^\times$.  If $\epsilon\not\in A^\times$,
  $O_K^\times/A^\times \cong \zmod{6}$, with the cyclic subgroup of
  order $3$ generated by $\bar{\epsilon}$.  Since $\epsilon\not\in B$,
  we must have $B^\times/A^\times\cong \zmod{2}$ in this case as well.
  Therefore,
  \begin{equation}
    \label{eq:55}
w(B_{1,2})=w(B_{1,4})=2.    
  \end{equation}
  Using $[O_K^\times: A^\times]=2\varpi$, we obtain
\begin{equation}
  \label{eq:56}
  h(B_{1,2})=\frac{1}{\varpi}\left(2-\left(\frac{2}{p}\right)\right)
  h(O_{K_1})
  \ \text{and}\ 
  h(B_{1,4})=\frac{2}{\varpi}\left(2-\left(\frac{2}{p}\right)\right)
  h(O_{K_1}). 
\end{equation}
\end{sect}

\begin{sect}
  Suppose that $K=K_3$.  By Exercise 42 of \cite[Chapter
  2]{MR0457396}, a $\zz$-basis of $\OO_{K_3}$ is
\begin{equation}
  \label{xeq:39}
  \left\{1,\quad \omega_p=\frac{1+\sqrt{p}}{2}, \quad \zeta_6=\frac{1+\sqrt{-3}}{2}, \quad
\omega_p\zeta_6=   \frac{(1+\sqrt{p} )(1+\sqrt{-3} )}{4}\right\}.
\end{equation}
Note that $2$ is inert in $L:=\qq(\zeta_6)=\qq(\sqrt{-3})\subset K$.
There are two primes $\p_1,\p_2$ above $2O_L$ in $K$. Both have
residue fields $O_K/\p_1\simeq
O_K/\p_2\simeq\ff_4$. Therefore, $O_L/2O_L\simeq \ff_4$ embeds
diagonally\footnote{Since the isomorphisms $O_K/\p_i\simeq \ff_4$ is
  \textit{not} canonical, the diagonal of $(O_K/\p_1)\times
  (O_K/\p_2)$ depends on the choice of $(O_K/\p_1)\simeq
  (O_K/\p_2)$. Here both of them are identified naturally with
  $O_L/2O_L$. In Section~\ref{subsec:the-B-3-2-order}, we 
   have a different diagonal. However, whichever diagonal we choose,
 the prime field $A/2O_F\cong \ff_2$ embeds canonically in it.} into
  \begin{equation}
    \label{eq:3}
    O_K/2O_K\cong (O_K/\p_1)\times
  (O_K/\p_2)\simeq \ff_4\times \ff_4. 
  \end{equation}

  Suppose that $B\supseteq B_{3,4}:=\zz[\sqrt{p}, \zeta_6]$.  Since
  $B_{3,4}/2O_K$ is a $2$-dimensional $\ff_2$-vector space spanned by
  the images of $1$ and $\zeta_6$, we have a canonical isomorphism
  $B_{3,4}/2O_K\cong O_L/2O_L$. The only other subring of $\ff_4\times
  \ff_4$ containing the diagonal is $\ff_4\times \ff_4$ itself. It
  follows that $B_{3,4}$ is the only proper $A$-order in $K$
  containing $\zeta_6$.

  We calculate the class number of $B_{3,4}$ using
  (\ref{xeq:37}) with $\a=2O_K$. It has already been shown that
  $(B_{3,4}/2O_K)^\times \simeq \ff_4^\times \simeq \zmod{3}$, and 
  \begin{equation}
    \label{eq:2}
    (O_K/2O_K)^\times \cong (O_K/\p_1)^\times\times
  (O_K/\p_2)^\times\simeq (\zmod{3})^2.
  \end{equation}
If $\epsilon\in A$, then $O_K^\times =
  B_{3,4}^\times$; otherwise, $O_K^\times/B_{3,4}^\times$ is a
  cyclic group of order $3$, generated by the image of $\epsilon$. It
  follows that 
  \begin{equation}
    \label{xeq:44}
w(B_{3,4})=3,\qquad    h(B_{3,4})=\frac{3h(O_{K_3})}{\varpi}=
    \begin{cases}
      3h(O_{K_3}) &\qquad \text{ if } \epsilon \in A;\\
       h(O_{K_3}) &\qquad \text{ if } \epsilon \not\in A.\\
    \end{cases}
  \end{equation}
\end{sect}

\begin{sect}
  Suppose that $K=K_3=\qq(\sqrt{p}, \sqrt{-3} )$, and $\varpi=1$. In
  other words, we assume $\epsilon\in A^\times$ and
  $O_F^\times=A^\times$. For example, this is the case 
  if $p\equiv 1\pmod{8}$ by Lemma~\ref{lem:p-equiv-1mod8-epsilon}. 
  For any quadratic proper 
  $A$-order $B$ with $w(B)>1$, we have
\[\{1\}\varsubsetneq B^\times/A^\times \subseteq 
O_K^\times/A^\times \simeq  \zmod{3}. \] 
Hence, $B^\times = O_K^\times$, and $B\supseteq \zz[\sqrt{p},
  \zeta_6]$. It follows that $B_{3,4}$ is the only proper
  $A$-order with $w(B)>1$ in this case. 
\end{sect}

\begin{sect}\label{subsec:the-B-3-2-order}
  Suppose that $K=K_3=\qq(\sqrt{p}, \sqrt{-3} )$, and $\varpi=3$. By
  an abuse of notation, we still write $\epsilon$ and $\zeta_6$ for
  their images in $O_K^\times/A^\times$.  Then
\[\{1\}\varsubsetneq B^\times/A^\times \subseteq O_K^\times/A^\times=
\dangle{\epsilon, \zeta_6} \simeq
  (\zmod{3})^2.\]
Since $\epsilon\not\in B$, $B^\times/A^\times$ is one of the following
cyclic subgroup of order $3$ in $O_K^\times/A^\times$:
$\dangle{\epsilon\zeta_6},  \dangle{\epsilon\zeta_6^{-1}}, 
\dangle{\zeta_6}.$ 

The case $B\ni \zeta_6$ has already been treated in the previous
subsections. So we focus on the orders
\[B_{3,2}:=A[\epsilon\zeta_6]=\zz[\sqrt{p}, \epsilon\zeta_6], 
\qquad B_{3,2}':=A[\epsilon\zeta_6^{-1}]=
\zz[\sqrt{p}, \epsilon\zeta_6^{-1}]. \]
Clearly $B_{3,2}'$ coincides with the complex conjugation of
$B_{3,2}$. 

Since $(\epsilon\zeta_6)^3=-\epsilon^3\in A$, the order $B_{3,2}$ is
generated as a $A$-module by the set $\{1, \epsilon\zeta_6,
\epsilon^2\zeta_6^2\}$. We claim that $B_{3,2}\supset 2O_K$. A
$\zz$-basis of $O_K$ is given in (\ref{xeq:39}). Clearly, $2\in A$ and
$2\omega_p\in A$ with $\omega_p=(1+\sqrt{p})/2$. Let
$a=\Tr_{F/\qq}(\epsilon)$ and recall that $\Nm_{F/\qq}(\epsilon)=-1$,
we have $\epsilon^2=a\epsilon+1$. Therefore,
\[\epsilon^2\zeta_6^2=(a\epsilon+1)
(\zeta_6-1)=a\epsilon\zeta_6+\zeta_6-a\epsilon-1.
\] 
It follows that $B_{3,2}$ is also generated over $A$ by 
$\{1, \epsilon\zeta_6, \zeta_6-a\epsilon\}$. 
Since $2a\epsilon\in A$, we have
$2\zeta_6= 2(\zeta_6-a\epsilon)+2a\epsilon\in B_{3,2}$. 
Lastly, we
need to show that $2\omega_p\zeta_6\in B_{3,2}$. Since
$\epsilon\not\in A$, there exists $x\in A$ such that
$\epsilon=x+\omega_p$. Note that $2x\zeta_6\in B_{3,2}$ because
$2\zeta_6\in B_{3,2}$, so
$2\omega_p\zeta_6=2(\epsilon-x)\zeta_6=2\epsilon\zeta_6-2x\zeta_6\in
B_{3,2}$. This finishes the proof of our claim. 

Next, we show that $B_{3,2}$ and $B_{3,2}'$ are indeed proper
$A$-orders and calculate their class numbers. Since $p\equiv 5
\pmod{8}$, we have $O_F/2O_F\simeq \ff_4$, which is generated by the
image of $\epsilon$ over $A/2O_F\cong \ff_2$. Denote this image by
$\bar{\epsilon}$. Recall that $O_K=O_F[\zeta_6]$, so
\[O_K/2O_K\simeq \ff_4[t]/(t^2-t+1)\simeq \ff_4\times \ff_4,\] sending
$t\mapsto (\bar\epsilon, \bar\epsilon+1)$. One checks that
$B_{3,2}/2O_K=\ff_4\times \ff_2$, and $B_{3,2}'=\ff_2\times \ff_4$. In
particular, they do not contain the diagonal of $\ff_4\times \ff_4$,
which is identified with $O_F/2O_F$. Thus both $B_{3,2}$ and
$B_{3,2}'$ are proper $A$-orders of index $2$ in $O_K=O_{K_3}$, 
conforming with the convention of our notations. In particular,
\begin{equation}\label{eq:5}
w(B_{3,2})=w(B_{3,2}')=3.
\end{equation}
Using (\ref{xeq:37}), one
sees that
\begin{equation}
  \label{eq:6}
h(B_{3,2})=h(B_{3,2}')=h(O_{K_3}).
\end{equation}
\end{sect}

\section*{Acknowledgements}
The authors thank Markus Kirschmer and Yifan Yang for very helpful
discussions. The revision of the present manuscript is made during
CF Yu's stay at the Max-Planck-Institut f\"ur Mathematik. He is
grateful to the Institute for kind hospitality and 
excellent working conditions. 
The authors also thank the referee for his/her careful reading and 
helpful comments.
J.~Xue was partially supported by the grant NSC 102-2811-M-001-090. 
TC Yang and CF Yu are partially supported by the grants MoST
100-2628-M-001-006-MY4, 103-2811-M-001-142 and 103-2918-I-001-009.

\end{document}